\documentclass[reqno]{amsart}
\usepackage[backend=bibtex, sorting=none, bibstyle=alphabetic, citestyle=alphabetic, sorting=nyt,maxbibnames=99]{biblatex}  
\renewbibmacro{in:}{}
\bibliography{ref.bib}
\usepackage{amssymb}
\usepackage{calrsfs}
\usepackage{graphics,graphicx}
\usepackage{enumerate}
\usepackage{enumitem}
\usepackage{url}
\usepackage{xcolor}
\usepackage[ruled, lined, linesnumbered, longend]{algorithm2e}
\usepackage{hyperref}

\newtheorem{theorem}{Theorem}[section]
\newtheorem{lemma}[theorem]{Lemma}
\newtheorem{question}[theorem]{Question}
\newtheorem{proposition}[theorem]{Proposition}

\newtheorem{corollary}[theorem]{Corollary}
\numberwithin{equation}{section}

\theoremstyle{remark}
\newtheorem*{remark}{Remark}

\DeclareMathOperator{\N}{\mathbb{N}}

\def\reals{\hbox{\rm I\kern-.18em R}}
\def\complexes{\hbox{\rm C\kern-.43em
\vrule depth 0ex height 1.4ex width .05em\kern.41em}}
\def\field{\hbox{\rm I\kern-.18em F}} 

\allowdisplaybreaks

\begin{document}

\title[Sum of a prime and a square-free number with divisibility conditions]{On the sum of a prime and a square-free number with divisibility conditions}

\author{Shehzad Hathi and Daniel R. Johnston}
\address{School of Science, The University of New South Wales, Canberra, Australia}
\email{s.hathi@student.adfa.edu.au}
\address{School of Science, The University of New South Wales, Canberra, Australia}
\email{daniel.johnston@student.adfa.edu.au}
\date\today
\keywords{}

\begin{abstract}
    Every integer greater than two can be expressed as the sum of a prime and a square-free number. Expanding on recent work, we provide explicit and asymptotic results when divisibility conditions are imposed on the square-free number. For example, we show for odd $k\leq 10^5$ and even $k\leq 2\cdot 10^5$ that any even integer $n\geq 40$ can be expressed as the sum of a prime and a square-free number coprime to $k$. We also discuss applications to other Goldbach-like problems.
\end{abstract}

\maketitle

\section{Introduction}\label{sectintro}
\subsection{Motivation}
A standard variant of Goldbach's conjecture is to consider representations of integers $n$ as 
\begin{equation*}
    n=p+\eta
\end{equation*}
where $p$ is a prime and $\eta$ belongs to a set $S$ containing the primes.\\
In the case where $S$ is the set of primes and semiprimes\footnote{Here, a semiprime refers to a positive integer with exactly two prime factors.}, we have the following theorem due to Chen.
\begin{theorem}[Chen]\label{chenthm}
    Every sufficiently large even integer can be represented as the sum of two primes, or a prime and a semiprime.
\end{theorem}
Recently, Chen's theorem has been shown to hold for $n\geq\exp(\exp(32.6))$ (see \cite{BJV22}), and is conjectured to hold for all $n\geq 4$. The most extensive verification of Chen's theorem to date is due to Oliveira e Silva, Herzog and Pardi \cite{e2014empirical} who verified Goldbach's conjecture (and thus Chen's Theorem) for $n\leq 4\cdot 10^{18}$. Given that this verification took 770 one-core CPU years to finish, verifying Chen's Theorem for all even $n\leq\exp(\exp(32.6))\approx 10^{4\cdot 10^{14}}$ seems computationally impractical at present.

In 2017, Dudek \cite{dudek2017sum} proved the following theorem, which makes explicit an earlier result of Estermann \cite[Theorem 3]{estermann1931representations}.
\begin{theorem}[Dudek]\label{dudekthm}
    Every integer greater than 2 can be represented as the sum of a prime and a square-free number.
\end{theorem}
For sufficiently large even $n$, Dudek's theorem is weaker than Chen's Theorem. However, Dudek's theorem has the advantage of being completely explicit, in the sense that we know every value of $n$ for which it holds.

Motivated by theorems \ref{chenthm} and \ref{dudekthm}, we investigate the following questions.

\begin{question}\label{squarefreeq}
    What happens if we place divisibility conditions on the square-free number in Dudek's theorem? That is, if $k$ is a fixed integer, can we find an $n_0$ such that every $n\geq n_0$ can be expressed as the sum of a prime and a square-free number coprime to $k$? 
\end{question}

\begin{question}\label{chenq}
    How much does one need to weaken Chen's theorem in order to gain a completely explicit result? In particular, can we find a $K>2$ such that every even integer $n\geq 4$ can be expressed as the sum of a prime and a number with at most $K$ prime factors? 
\end{question}

\subsection{Statement of main results}\label{statementres}
In response to Question \ref{squarefreeq}, we prove the following results.
\begin{theorem}\label{firsteventhm}
    Let $k$ be even with $2\leq k\leq 2\cdot 10^5$ or odd with $1\leq k\leq 10^5$. Then every even $n\geq 40$ can be expressed as the sum of a prime and a square-free number $\eta>1$ with $(\eta,k)=1$.
\end{theorem}
\begin{theorem}\label{firstoddthm}
    Let $k$ be odd with at most 2 prime factors and $1\leq k\leq 10^5$. Then every $n\geq 36$ can be expressed as the sum of a prime and a square-free number $\eta>1$ with $(\eta,k)=1$.
\end{theorem}
These results improve on recent work of Francis and Lee \cite{francis2020additive}, who only considered the case where $k$ is prime.

Importantly, we note that Theorem \ref{firsteventhm} holds only for even $n$ whereas Theorem \ref{firstoddthm} holds for both odd and even $n$. If we remove the restriction on the prime factors of $k$ in Theorem \ref{firstoddthm} then the result holds for all $n\geq 10^{25}$ (see Theorem \ref{bigoddthm}).

In addition, we also provide asymptotic formulas for the (logarithmically-weighted) number of representations of an integer $n$ as the sum of a prime and a square-free number coprime to $k$ (see Theorem \ref{asympthm}).

The main focus of this paper will be on proving Theorems \ref{firsteventhm} and \ref{firstoddthm}. 

An answer to Question \ref{chenq} has been deferred to follow up work by the second author and V. Starichkova \cite{johnstonsome22}. In particular, we have the following result.

\begin{theorem}[{\cite[Theorem 1.3]{johnstonsome22}}]\label{kthm}
    Every even integer $n\geq 4$ can be expressed as the sum of a prime and a number with at most $K=369$ prime factors.
\end{theorem}

To prove Theorem \ref{kthm} one uses an explicit version of the linear sieve given in \cite[\S 2]{BJV22} for large $n$, followed by a simple application of Theorem \ref{firsteventhm} for small $n$ (see \cite[Lemma 3.5]{johnstonsome22}).

A variant of Theorem \ref{firsteventhm} assuming the Generalised Riemann Hypothesis is also discussed in \cite[\S 5]{johnstonsome22}.

\subsection{Outline of paper}
The main idea in this paper is to use an inclusion-exclusion argument to generalise the results of \cite{dudek2017sum} and \cite{francis2020additive}. In particular, if $R_k(n)$ denotes the weighted number of representations of $n$ as the sum of a prime and a square-free number coprime to $k$, then we show that
\begin{equation*}
    |R_k(n)-A_k(n)|<E_k(n),
\end{equation*}
for some main term $A_k(n)$ and error term $E_k(n)$. For sufficiently large $n$, we then bound $E_k(n)$ so that $R_k(n)>0$. For smaller values of $n$, we use computations\footnote{Python code for the computations can be found here: \url{https://bit.ly/3m9QYlk}} to establish Theorems \ref{firsteventhm} and \ref{firstoddthm}.

The general outline is as follows. In Section \ref{sectnot}, we establish the main notation and ideas that will be used throughout the paper. In Sections \ref{sectbounds} and \ref{sectasym}, we establish some bounds and asymptotic formulas. Theorems \ref{firsteventhm} and \ref{firstoddthm} are then established in Section \ref{sectcomp}. Finally, in Section \ref{sectext}, we discuss an extension of our results whereby one places further conditions on the square-free number.

\section{Setup and Notation}\label{sectnot}
Following Dudek \cite{dudek2017sum}, we define
\begin{equation*}
    R(n):=\sum_{p\leq n}\mu^2(n-p)\log(p)
\end{equation*}
to be the logarithmically-weighted number of representations of $n$ as the sum of a prime $p$ and a square-free number. Here, $\mu$ is the Möbius function whose square indicates whether its argument is square-free or not. Now, fix a square-free integer $k>0$. We wish to obtain bounds for 
\begin{equation*}
    R_k(n):=\sum_{\substack{p\leq n\\(n-p,k)=1}}\mu^2(n-p)\log(p).
\end{equation*}
In particular, $R_k(n)$ is the number of representations of $n$ as the sum of a prime and a square-free number coprime to $k$.

By \cite[Chapter 2 Exercise 6]{apostol1998introduction} we have
\begin{equation}\label{musquaredeq}
    \mu^2(n)=\sum_{a^2\mid n}\mu(a).
\end{equation}
Using \eqref{musquaredeq} and an inclusion-exclusion argument, we see that
\begin{align*}
    R_k(n)&=R(n)-\sum_{\substack{p\leq n\\(n-p,k)>1}}\mu^2(n-p)\log(p)\\
    &=\sum_{d\mid k}\mu(d)\sum_{\substack{p\leq n\\p\equiv n\: (d)}}\mu^2(n-p)\log(p)\\
    &=\sum_{d\mid k}\mu(d)\sum_{a\leq n^{\frac{1}{2}}}\mu(a)\sum_{\substack{p\leq n\\p\equiv n\: (a^2)\\p\equiv n\: (d)}}\log(p).
\end{align*}
Applying the Chinese remainder theorem then gives
\begin{equation}\label{rksimpeq}
    R_k(n)=\sum_{d\mid k}\mu(d)\sum_{e\mid d}\sum_{\substack{a\leq n^{\frac{1}{2}}\\(a,d)=e}}\mu(a)\theta\left(n;da^2/e,n\right),
\end{equation}
where for any $m>0$,
\begin{equation*}
    \theta\left(x;m,n\right)=\sum_{\substack{p\leq x\\p\equiv n\:(m)}}\log(p).
\end{equation*}
For positive square-free integers $e$ and $d$ with $e\mid d$, we thus define
\begin{equation*}
    R_{d,e}(n):=\sum_{\substack{a\leq n^{\frac{1}{2}}\\(a,d)=e}}\mu(a)\theta\left(n;da^2/e,n\right)
\end{equation*}
and try to bound $R_{d,e}(n)$. When $(n,m)=1$, we will use inequalities of the form
\begin{equation}\label{bennettbounds}
    \left|\theta(n;m,n)-\frac{n}{\varphi(m)}\right|<\frac{c_{\theta}(m)n}{\log(n)},
\end{equation}
where $c_{\theta}(m)$ is a constant depending on $m$. For $m\geq 3$ we take $c_{\theta}(m)$ from \cite{bennett2018explicit} and the values of $n$ for which they apply. For $m=1$, we use \cite[Table 15]{broadbent2021sharper}. The case $m=2$ will not be required for our computations.

\section{Some bounds}\label{sectbounds}
As motivated in the previous section, we now provide bounds for $R_{d,e}(n)$. This will be done using a method that generalises that in \cite{dudek2017sum} and \cite{francis2020additive}. In particular, \cite[Lemma 9]{francis2020additive} corresponds to the case $d=e=1$.
\begin{lemma}\label{delemma}
    Let $e$ and $d$ be positive square-free integers with $e\mid d$. Define
    \begin{equation*}
        A_{d,e}(n):=cn\prod_{p\mid n}\left(1+\frac{1}{p^2-p-1}\right)\frac{\mu(e)d/e}{\prod_{q\mid d}(q^2-q-1)}.
    \end{equation*}
    Here,
    \begin{equation*}
        c=\prod_p\left(1-\frac{1}{p(p-1)}\right)\approx 0.37395
    \end{equation*} 
    is Artin's constant, and $p$ and $q$ are primes. Fix $N>0$ and $C\in(0,1/2)$. Then, for sufficiently large $n$, say $n\geq n_0$, with $(d,n)=1$,
    \begin{align*}
        |R_{d,e}(n)-A_{d,e}(n)|<&\frac{n}{\log(n)}\sum_{\substack{a\leq\sqrt{N}\sqrt{e/d}\\(a,d)=e}}c_{\theta}\left(da^2/e\right)\mu^2(a)\\
        &+n\left(\frac{1+2C}{1-2C}\right)\frac{1}{\varphi\left(d/e\right)}\sum_{\substack{a>\sqrt{N}\sqrt{e/d}\\(a,d)=e}}\frac{\mu^2(a)}{\varphi(a^2)}\\
        &+n\log(n)\left(n^{-\frac{1}{2}}\left(\frac{1}{e}-\frac{1}{d}\right)+\frac{1}{\sqrt{de}}n^{-C}+n^{-2C}\right)=:E_{d,e}(n).
    \end{align*}
    In particular, we can take $n_0$ to be any integer with $n_0^C>\sqrt{N}$ and such that \eqref{bennettbounds} holds for each value of $c_{\theta}(da^2/e)$ used.
\end{lemma}
\begin{proof}
    We will only prove $R_{d,e}(n)<A_{d,e}(n)+E_{d,e}(n)$ since the other inequality $R_{d,e}(n)>A_{d,e}(n)-E_{d,e}(n)$ follows from almost identical reasoning.
    
    First note that if $(a,n)>1$ then $\theta\left(n;da^2/e,n\right)\leq\log(n)$. Moreover, if $(a,d)=e$ then $e\mid a$ so there are less than $(n^{C}\sqrt{e/d})/e=n^C/\sqrt{de}$ values of $a\leq n^C\sqrt{e/d}$ satisfying $(a,d)=e$. Hence,
    \begin{align*}
        R_{d,e}(n)\leq&\sum_{\substack{a\leq n^{C}\sqrt{e/d}\\(a,d)=e\\(a,n)=1}}\mu(a)\theta\left(n;da^2/e,n\right)+\frac{n^C}{\sqrt{de}}\log(n)\\
        &\qquad+\sum_{\substack{n^C\sqrt{e/d}<a\leq n^{\frac{1}{2}}\\(a,d)=e}}\mu(a)\theta\left(n;da^2/e,n\right).
    \end{align*}
    We write
    \begin{equation}\label{maineqnde}
        R_{d,e}(n)\leq\Sigma_1+\Sigma_2+\Sigma_3+\frac{n^C}{\sqrt{de}}\log(n),
    \end{equation}
    where
    \begin{align*}
        \Sigma_1&=\sum_{\substack{a\leq \sqrt{N}\sqrt{e/d}\\(a,d)=e\\(a,n)=1}}\mu(a)\theta\left(n;da^2/e,n\right),\\
        \Sigma_2&=\sum_{\substack{\sqrt{N}\sqrt{e/d}<a\leq n^{C}\sqrt{e/d}\\(a,d)=e\\(a,n)=1}}\mu(a)\theta\left(n;da^2/e,n\right),\ \text{and}\\
        \Sigma_3&=\sum_{\substack{n^C\sqrt{e/d}< a\leq n^{\frac{1}{2}}\\(a,d)=e}}\mu(a)\theta\left(n;da^2/e,n\right).
    \end{align*}
    Using \eqref{bennettbounds} for $a\leq\sqrt{N}\sqrt{e/d}$, we have
    \begin{align*}
        \Sigma_1&<n\left(\sum_{\substack{a\leq\sqrt{N}\sqrt{e/d}\\(a,d)=e\\(a,n)=1}}\frac{\mu(a)}{\varphi\left(da^2/e\right)}+\sum_{\substack{a\leq\sqrt{N}\sqrt{e/d}\\(a,d)=e}}\frac{c_{\theta}\left(da^2/e\right)\mu^2(a)}{\log(n)}\right)\\
        &=n\left(\sum_{\substack{(a,d)=e\\(a,n)=1}}\frac{\mu(a)}{\varphi\left(da^2/e\right)}-\sum_{\substack{a>\sqrt{N}\sqrt{e/d}\\(a,d)=e\\(a,n)=1}}\frac{\mu(a)}{\varphi\left(da^2/e\right)}+\sum_{\substack{a\leq\sqrt{N}\sqrt{e/d}\\(a,d)=e}}\frac{c_{\theta}\left(da^2/e\right)\mu^2(a)}{\log(n)}\right)
    \end{align*}
    For the first term, we write $a=eb$ where $(b,d)=(b,e)=1$. Then,
    \begin{align*}
        n\sum_{\substack{(a,d)=e\\(a,n)=1}}\frac{\mu(a)}{\varphi\left(da^2/e\right)}&=\frac{n}{\varphi(d/e)}\frac{\mu(e)}{\varphi(e^2)}\sum_{\substack{(b,d)=1\\(eb,n)=1}}\frac{\mu(b)}{\varphi(b^2)}\\
        &=n\frac{\mu(e)}{\varphi(de)}\sum_{\substack{(b,d)=1\\(eb,n)=1}}\frac{\mu(b)}{\varphi(b^2)}.
    \end{align*}
    Since $(d,n)=(e,n)=1$, this is equivalent to
    \begin{align*}
        n\frac{\mu(e)}{\varphi(de)}\prod_{\substack{p\nmid n \\ p\nmid d}}\left(1-\frac{1}{\varphi(p^2)}\right)&=n\frac{\mu(e)}{\varphi(de)}\frac{\prod_{p}\left(1-\frac{1}{\varphi(p^2)}\right)}{\prod_{p\mid n}\left(1-\frac{1}{\varphi(p^2)}\right)\prod_{q\mid d}\left(1-\frac{1}{\varphi(q^2)}\right)}
    \end{align*}
    which simplifies to $A_{d,e}(n)$. Hence, 
    \begin{equation*}
        \Sigma_1<A_{d,e}(n)+n\left(-\sum_{\substack{a>\sqrt{N}\sqrt{e/d}\\(a,d)=e\\(a,n)=1}}\frac{\mu(a)}{\varphi\left(da^2/e\right)}+\frac{1}{\log(n)}\sum_{\substack{a\leq\sqrt{N}\sqrt{e/d}\\(a,d)=e}}c_{\theta}\left(da^2/e\right)\mu^2(a)\right).
    \end{equation*}
    Next we move onto $\Sigma_2$. Similarly to Dudek \cite{dudek2017sum}, we use the explicit Brun--Titchmarsh theorem \cite[Theorem 2]{montgomery1973large} to obtain
    \begin{equation*}
        \theta\left(n;da^2/e,n\right)=\frac{n}{\varphi\left(da^2/e\right)}+\varepsilon\left(\frac{1+2C}{1-2C}\right)\frac{n}{\varphi\left(da^2/e\right)}
    \end{equation*}
    for some $\varepsilon$ with $|\varepsilon|<1$. As a result,
    \begin{equation*}
        \Sigma_2<n\left(\sum_{\substack{\sqrt{N}\sqrt{e/d}<a\leq n^{C}\sqrt{e/d}\\(a,d)=e\\(a,n)=1}}\frac{\mu(a)}{\varphi\left(da^2/e\right)}+\left(\frac{1+2C}{1-2C}\right)\sum_{\substack{\sqrt{N}\sqrt{e/d}<a\leq n^{C}\sqrt{e/d}\\(a,d)=e\\(a,n)=1}}\frac{\mu^2(a)}{\varphi\left(da^2/e\right)}\right).
    \end{equation*}
    Thus, noting that $\varphi(da^2/e)=\varphi(d/e)\varphi(a^2)$ if $(a,d)=e$,
    \begin{align*}
        \Sigma_1+\Sigma_2&<A_{d,e}(n)+n\left(\frac{1}{\log(n)}\sum_{\substack{a\leq\sqrt{N}\sqrt{e/d}\\(a,d)=e}}c_{\theta}\left(da^2/e\right)\mu^2(a)-\sum_{\substack{a>n^C\sqrt{e/d}\\(a,d)=e\\(a,n)=1}}\frac{\mu(a)}{\varphi(da^2/e)}\right.\\
        &\qquad\qquad\qquad\qquad+\left.\left(\frac{1+2C}{1-2C}\right)\frac{1}{\varphi\left(d/e\right)}\sum_{\substack{\sqrt{N}\sqrt{e/d}<a\leq n^{C}\sqrt{e/d}\\(a,d)=e\\(a,n)=1}}\frac{\mu^2(a)}{\varphi(a^2)}\right)\\
        &<A_{d,e}(n)+n\left(\frac{1}{\log(n)}\sum_{\substack{a\leq\sqrt{N}\sqrt{e/d}\\(a,d)=e}}c_{\theta}\left(da^2/e\right)\mu^2(a)\right.\\
        &\qquad\qquad\qquad\qquad+\left.\left(\frac{1+2C}{1-2C}\right)\frac{1}{\varphi\left(d/e\right)}\sum_{\substack{a>\sqrt{N}\sqrt{e/d}\\(a,d)=e}}\frac{\mu^2(a)}{\varphi(a^2)}\right).
    \end{align*}
    Finally, for $\Sigma_3$, we use the substitution $\ell=a/e$, and the trivial bound $\theta(x;m,n)\leq (1+\frac{x}{m})\log x$ to obtain
    \begin{align*}
        |\Sigma_3|&\leq\sum_{\substack{n^C\sqrt{e/d}< a\leq n^{\frac{1}{2}}\\(a,d)=e}}\theta\left(n;da^2/e,n\right)\notag\\
        &\leq\sum_{n^C/\sqrt{de}< \ell\leq n^{\frac{1}{2}}/e}\theta\left(n;\ell^2de,n\right)\notag\\
        &\leq\sum_{n^C/\sqrt{de}< \ell\leq n^{\frac{1}{2}}/e}\left(1+\frac{n}{\ell^2de}\right)\log(n)\notag\\
        &\leq\frac{n^{\frac{1}{2}}\log(n)}{e}-\frac{n^C}{\sqrt{de}}\log(n)+\frac{n\log(n)}{de}\sum_{n^C/\sqrt{de}< \ell\leq n^{\frac{1}{2}}/e}\frac{1}{\ell^2}\notag\\
        &\leq n^{\frac{1}{2}}\log(n)\left(\frac{1}{e}-\frac{1}{d}\right)-\frac{n^C}{\sqrt{de}}\log(n)+\frac{n^{1-C}}{\sqrt{de}}\log(n)+n^{1-2C}\log(n),
    \end{align*}
    where the last inequality follows after a standard partial summation argument.\\[8pt]
    Substituting our inequalities for $\Sigma_1$, $\Sigma_2$ and $\Sigma_3$ into \eqref{maineqnde} gives $R_{d,e}(n)<A_{d,e}(n)+E_{d,e}(n)$ as desired.
\end{proof}

We can now get bounds on $R_k(n)$.
\begin{proposition}\label{propcoprime}
    Let $k>0$ be a square-free integer and define
    \begin{equation*}
        A_k(n):=cn\prod_{p\mid n}\left(1+\frac{1}{p^2-p-1}\right)\prod_{q\mid k}\left(1-\frac{q-1}{q^2-q-1}\right).
    \end{equation*}
    Keeping the notation of Lemma \ref{delemma}, if $(k,n)=1$ and $n\geq n_0$ for each choice of $e$ and $d$ with $e\mid d\mid k$, then
    \begin{equation*}
        \left|R_k(n)-A_k(n)\right|<\sum_{d\mid k}\sum_{e\mid d}E_{d,e}(n)=:E_k(n).
    \end{equation*}
\end{proposition}
\begin{proof}
    We will only prove $R_k(n)>A_k(n)-E_k(n)$ since the other inequality $R_k(n)<A_k(n)+E_k(n)$ follows from almost identical reasoning. Now, because $(k,n)=1$ we have $(d,n)=1$ for each $d\mid k$. Thus, by Lemma \ref{delemma} and \eqref{rksimpeq}
    \begin{align*}
        R_k(n)&=\sum_{d\mid k}\mu(d)\sum_{e\mid d}R_{d,e}(n)\\
        &>\sum_{d\mid k}\mu(d)\sum_{e\mid d}A_{d,e}(n)-\sum_{d\mid k}\sum_{e\mid d}E_{d,e}(n)\\
        &=cn\prod_{p\mid n}\left(1+\frac{1}{p^2-p-1}\right)\sum_{d\mid k}\mu(d)\sum_{e\mid d}\frac{\mu(e)d/e}{\prod_{q\mid d}(q^2-q-1)}-E_k(n).
    \end{align*}
    To complete the proof, we then note $\varphi(d)=\sum_{e\mid d}\mu(e)d/e$ so that
    \begin{align*}
        \sum_{d\mid k}\mu(d)\sum_{e\mid d}\frac{\mu(e)d/e}{\prod_{q\mid d}(q^2-q-1)}&=\sum_{d\mid k}\frac{\mu(d)\varphi(d)}{\prod_{q\mid d}(q^2-q-1)}\\
        &=\prod_{q\mid k}\left(1-\frac{\varphi(q)}{q^2-q-1}\right)\\
        &=\prod_{q\mid k}\left(1-\frac{q-1}{q^2-q-1}\right)\qedhere
    \end{align*}
\end{proof}
\begin{proposition}\label{propnocoprime}
    Keeping the notation from Proposition \ref{propcoprime}, suppose now that $k$ is not necessarily coprime to $n$. Writing $k_n=k/(k,n)$ and letting $n\geq n_0$ for each choice of $e$ and $d$ with $e\mid d\mid k_n$, we have
    \begin{equation}\label{logkneq}
        |R_k(n)-A_{k_n}(n)|<E_{k_n}(n)+\log((k,n)).
    \end{equation}
\end{proposition}
\begin{proof}
    If $(k,n)=1$, \eqref{logkneq} is precisely Proposition \ref{propcoprime}. Now suppose that $(k,n)>1$. Since $k_n\mid k$, it follows that
    \begin{equation}\label{rkupper}
        R_k(n)\leq R_{k_n}(n)<A_{k_n}(n)+E_{k_n}(n)<A_{k_n}(n)+E_{k_n}(n)+\log((k,n)).
    \end{equation}
    Recall that $R_{k_n}(n)$ is the logarithmically-weighted number of representations of $n$ as
    \begin{equation*}
        n=p+\eta
    \end{equation*}
    where $p$ is prime and $\eta$ is square-free and coprime to $k_n$. If $p$ is coprime to $(k,n)=k/k_n$ then since $(k,n)\mid n$, we have that $\eta$ must also be coprime to $(k,n)$ and thus to $k$. Hence,
    \begin{align}\label{rklower}
        R_{k}(n)&>R_{k_n}(n)-\sum_{p\mid (k,n)}\log(p)\notag\\
        &=R_{k_n}(n)-\log((k,n))\notag\\
        &>A_{k_n}(n)-E_{k_n}(n)-\log((k,n)).
    \end{align}
    Combining \eqref{rkupper} and \eqref{rklower} gives \eqref{logkneq}.
\end{proof}

\section{Some asymptotics}\label{sectasym}

We would like $R_k(n)>0$ for sufficiently large $n$. However, this is not possible if $k$ is even. If $k$ is even, we are looking at representations $n=p+\eta$ where $p$ is prime and $\eta$ is an odd square-free number. If $n$ is odd, then $p=2$ and $\eta=n-2$. But there are infinitely many odd integers $n$ for which $n-2$ is not square-free.

Besides this caveat, we otherwise get positive asymptotic expressions for $R_k(n)$.
\begin{theorem}\label{asympthm}
    Suppose that $k$ is odd. Using the notation from Proposition \ref{propnocoprime}, we have, as $n\to\infty$
    \begin{equation*}
        \frac{R_k(n)}{n}\sim \frac{A_{k_n}(n)}{n}=c\prod_{p\mid n}\left(1+\frac{1}{p^2-p-1}\right)\prod_{q\mid k_n}\left(1-\frac{q-1}{q^2-q-1}\right)>0.
    \end{equation*}
    And if $k$ is even,
    \begin{equation*}
        \frac{R_k(2n)}{2n}\sim \frac{A_{k_{2n}}(2n)}{2n}=c\prod_{p\mid 2n}\left(1+\frac{1}{p^2-p-1}\right)\prod_{q\mid k_{2n}}\left(1-\frac{q-1}{q^2-q-1}\right)>0.
    \end{equation*}
\end{theorem}
\begin{proof}
    We begin with the case where $k$ is odd. Since $\sum_{p}\frac{1}{p^2-p-1}$ converges,
    \begin{equation}\label{nbounds}
        1\leq\prod_{p\mid n}\left(1+\frac{1}{p^2-p-1}\right)< K
    \end{equation}
    for some constant $K>1$. Moreover,
    \begin{equation}\label{kbounds}
        0<\prod_{q\mid k}\left(1-\frac{q-1}{q^2-q-1}\right)\leq\prod_{q\mid k_n}\left(1-\frac{q-1}{q^2-q-1}\right)\leq 1
    \end{equation}
    Combining \eqref{nbounds} and \eqref{kbounds} we see that $C_1n<A_{k_n}(n)<C_2n$, where
    \begin{align*}
        C_1&=c\prod_{q\mid k}\left(1-\frac{q-1}{q^2-q-1}\right),\\
        C_2&=cK
    \end{align*}
    are positive constants (with $C_1$ depending on $k$). Hence it suffices to prove that for any $\varepsilon>0$,
    \begin{equation}\label{epsiloneq}
        \frac{|R_k(n)-A_{k_n}(n)|}{n}<\varepsilon,
    \end{equation}
    holds for sufficiently large $n$. By Proposition \ref{propnocoprime} we have
    \begin{equation*}
        |R_k(n)-A_{k_n}(n)|<E_{k_n}(n)+\log((k,n)).
    \end{equation*}
    Now, $\log((k,n))\leq\log(k)=o(n)$. Then, $E_{k_n}(n)=\sum_{d\mid k_n}\sum_{e\mid d}E_{d,e}(n)$ so it suffices to show that $E_{d,e}(n)/n$ can be made arbitrarily small as $n\to\infty$. With reference to Lemma \ref{delemma}, we fix $C$ and note that the term
    \begin{align*}
        \left(\frac{1+2C}{1-2C}\right)\frac{1}{\varphi\left(d/e\right)}\sum_{\substack{a>\sqrt{N}\sqrt{e/d}\\(a,d)=e}}\frac{\mu^2(a)}{\varphi(a^2)}
    \end{align*}
    appearing in $E_{d,e}(n)/n$ can be made arbitrarily small by setting $N$ to be arbitrarily large. The result then follows since all other terms in $E_{d,e}(n)$ are $o(n)$.
    
    The argument for even $k$ is almost identical.
\end{proof}

\begin{corollary}
    If $k$ is odd, there exists $n_k>0$ such that $R_k(n)>0$ for all $n>n_k$.
\end{corollary}
\begin{corollary}
    If $k$ is even, there exists $n_k>0$ such that $R_k(n)>0$ for all \textbf{even} $n>n_k$.
\end{corollary}

\section{Computations and results}\label{sectcomp}
We now move on to our main computations and results. The code for this section can be found at \url{https://bit.ly/3m9QYlk}.

For our main results, Theorems \ref{firsteventhm} and \ref{firstoddthm}, we wish to exclude the square-free number $\eta=1$ from our representations. We thus define $\overline{R}_k(n)$ to be the logarithmically-weighted number of representations of $n$ as $n=p+\eta$ where $p$ is a prime, $\eta$ is a square-free number coprime to $k$ \textbf{and} $\eta\neq 1$. More precisely,
\begin{equation*}
    \overline{R}_k(n):=\sum_{\substack{p\leq n\\(n-p,k)=1\\p\neq n-1}}\mu^2(n-p)\log(p).
\end{equation*}
Note that
\begin{equation}\label{overlineeq}
    \overline{R}_k(n)\geq R_k(n)-\log(n-1)> R_k(n)-\log(n).
\end{equation}
Together with Theorem \ref{asympthm}, \eqref{overlineeq} implies that $\overline{R}_k(n)\sim R_k(n)$ when $k$ is odd and $\overline{R}_k(2n)\sim R_k(2n)$ when $k$ is even.

For a specific value of $k$ we also define the exception set
\begin{equation*}
    S_k=
    \begin{cases}
        \{n\in\N\::\:\overline{R}_k(n)=0\},&\text{if $k$ is odd},\\
        \{n\in 2\N\::\:\overline{R}_k(n)=0\},&\text{if $k$ is even}.
    \end{cases}
\end{equation*}
For our computations we are mainly interested in finding the largest value in $S_k$.

\subsection{Results for $n$ even}\label{evenresults}
We begin with the case where $n$ is even. This turns out to be significantly easier due to the assistance we get from Oliveira e Silva, Herzog and Pardi's Goldbach verification up to $4\cdot10^{18}$ \cite{e2014empirical}.

Let $k$ be square-free and even with $2\leq k\leq 2\cdot 10^5$. This range of $k$ is chosen in regard to the limits of Bennett et al.'s \cite{bennett2018explicit} explicit results for primes in arithmetic progressions. In particular, if $m\leq 10^5$, then \cite[Theorem 1.2]{bennett2018explicit} gives a value for $c_{\theta}(m)$ that holds for any $x\geq 8\cdot 10^9$. Whereas if $m>10^5$, then \cite{bennett2018explicit} only gives a value for $c_{\theta}(m)$ that holds for very large values of $x$, namely $x\geq\exp(0.03\sqrt{m}\log(m)^3)$.

Note that since $k$ is even, $(\eta,k)=1$ implies $(\eta,k/2)=1$ for any $\eta\in\N$. That is, $\overline{R}_{k}(n)>0$ implies $\overline{R}_{k/2}(n)>0$. Therefore, our subsequent results for even $k\leq 2\cdot 10^5$ also give us information about representations with the square-free number coprime to odd numbers $k/2\leq 10^5$ (see Corollary \ref{evencor}).
\begin{lemma}\label{evencomp}
    Let $n\geq 4\cdot 10^{18}$ be even and $k$ be square-free and even with $2\leq k\leq 2\cdot 10^5$. Then for any $C\in(0,1/2)$ with $n^C>\sqrt{10^5}$,
    \begin{align}\label{eveneq}
        \frac{\overline{R_k}(n)}{n}>2c&\prod_{q\mid k/2}\left(1-\frac{q-1}{q^2-q-1}\right)-\frac{1}{\log(n)}\sum_{\substack{d\mid k/2}}\sum_{e\mid d}\sum_{\substack{a\leq\sqrt{10^5}\sqrt{e/d}\\(a,d)=e}}c_{\theta}\left(da^2/e\right)\mu^2(a)\notag\\
        &-\left(\frac{1+2C}{1-2C}\right)\sum_{d\mid k/2}\sum_{e\mid d}\frac{1}{\varphi\left(d/e\right)}\sum_{\substack{a>\sqrt{10^5}\sqrt{e/d}\\(a,d)=e}}\frac{\mu^2(a)}{\varphi(a^2)}\notag\\
        &-\log(n)\left(\sum_{d\mid k/2}\sum_{e\mid d}\left(n^{-\frac{1}{2}}\left(\frac{1}{e}-\frac{1}{d}\right)+\frac{1}{\sqrt{de}}n^{-C}+n^{-2C}\right)\right)\notag\\
        &-\frac{\log(k)}{n}-\frac{\log(n)}{n}.
    \end{align}
    Here, we take $c_{\theta}(1)$ to be $8.6315\cdot 10^{-7}$ as per \cite[Table 15]{broadbent2021sharper}.
\end{lemma}
\begin{proof}
    By Proposition \ref{propnocoprime},
    \begin{equation*}
        R_k(n)>A_{k/2}(n)-E_{k/2}(n)-\log((k,n))>A_{k/2}(n)-E_{k/2}(n)-\log(k).
    \end{equation*}
    To complete the proof of the lemma we use \eqref{overlineeq} and note that 
    \begin{align*}
        A_{k/2}(n)&=cn\prod_{p\mid n}\left(1+\frac{1}{p^2-p-1}\right)\prod_{q\mid k/2}\left(1-\frac{q-1}{q^2-q-1}\right)\\
        &>2cn\prod_{q\mid k/2}\left(1-\frac{q-1}{q^2-q-1}\right)
    \end{align*}
    since $n$ is even.
\end{proof}

\begin{theorem}\label{eventhm}
    Let $k$ be even and square-free with $2\leq k\leq 2\cdot 10^5$. Then for all even $n\geq 40$ we have $\overline{R}_k(n)>0$.
\end{theorem}
\begin{proof}
    Set $C=0.2$. Then for any even $k$ with $2\leq k\leq 2 \cdot 10^5$, the right-hand side of \eqref{eveneq} is positive for $n\geq 4\cdot 10^{18}$. Hence, $\overline{R}_k(n)>0$ for all even $n\geq 4\cdot 10^{18}$ and each $k$ in our range. For even $2<n<4\cdot 10^{18}$, the Goldbach verification \cite{e2014empirical} implies that $\overline{R}_k(n)>0$ except possibly when $n=q_1+q_2$ where $q_1$ and $q_2$ are both prime divisors of $k$ (not necessarily distinct). For each $k$ in question, we checked all of these possible exceptions and found that the largest exception occured when $k=24738$ and $n=38$. 
\end{proof}
\begin{remark}
    A value for $c=0.37395...$ (to 45 decimal places) can be found in \cite{wrench1961evaluation}. Moreover, for computing 
    \begin{equation*}
        \frac{1}{\varphi\left(d/e\right)}\sum_{\substack{a>\sqrt{10^5}\sqrt{e/d}\\(a,d)=e}}\frac{\mu^2(a)}{\varphi(a^2)}=\frac{1}{\varphi\left(d/e\right)}\sum_{(a,d)=e}\frac{\mu^2(a)}{\varphi(a^2)}-\frac{1}{\varphi\left(d/e\right)}\sum_{\substack{a\leq\sqrt{10^5}\sqrt{e/d}\\(a,d)=e}}\frac{\mu^2(a)}{\varphi(a^2)}
    \end{equation*}
    we note that
    \begin{align*}
        \frac{1}{\varphi\left(d/e\right)}\sum_{(a,d)=e}\frac{\mu^2(a)}{\varphi(a^2)}&=\frac{1}{\varphi\left(d/e\right)\varphi(e^2)}\frac{\prod_p\left(1+\frac{1}{\varphi(p^2)}\right)}{\prod_{p\mid d}\left(1+\frac{1}{\varphi(p^2)}\right)}\\
        &=\frac{\varphi(d^2)/\varphi(e^2)}{\varphi(d/e)\prod_{p\mid d}(p^2-p+1)}\sum_a\frac{\mu^2(a)}{\varphi(a^2)}\\
        &=\frac{d/e}{\prod_{p\mid d}(p^2-p+1)}\sum_a\frac{\mu^2(a)}{\varphi(a^2)}
    \end{align*}
    and
    \begin{equation*}
        \sum_a\frac{\mu^2(a)}{\varphi(a^2)}=\prod_p\left(1+\frac{1}{p^2-p}\right)=\prod_p\frac{(1-p^{-2})^{-1}(1-p^{-3})^{-1}}{(1-p^{-6})^{-1}}=\frac{\zeta(2)\zeta(3)}{\zeta(6)}.
    \end{equation*}
\end{remark}

\begin{corollary}\label{evencor}
    Let $k$ be odd and square-free with $1\leq k\leq 10^5$. Then for all even $n\geq 40$ we have $\overline{R}_k(n)>0$.
\end{corollary}
\begin{proof}
    The result follows from Theorem 5.2 since $\overline{R}_{2k}(n)>0$ implies $\overline{R}_{k}(n)>0$. Note that $n=38$ is an exception for $k=24738/2=12369$.
\end{proof}
Together, Theorem \ref{eventhm} and Corollary \ref{evencor} give Theorem \ref{firsteventhm}.

We now list the exception sets $S_k$ for some specific values of $k$. In particular, we consider the primorials $p_m\#=\prod_{i=1}^m p_i$, where $p_i$ is the $i^{\text{th}}$ prime number. The primorials are of interest since every square-free number divides a sufficiently large primorial and for any even $k\mid p_m\#$, we have $S_k\subseteq S_{p_m\#}$. Our result for $13\#=30030$ is also that which is used in the follow up work \cite{johnstonsome22}.

\begin{table}[h]
\centering
\caption{Exception sets for the first 6 primorials.}
\begin{tabular}{|c|c|}
\hline
$k$ & $S_k$ \\
\hline
$ 2 $&$ \{2,4\} $\\
\hline
$ 6 $&$ \{2,4,6\} $\\
\hline
$ 30 $&$ \{2,4,6,8\} $\\
\hline
$ 210 $&$ \{2,4,6,8,10,12\} $\\
\hline
$ 2310 $&$ \{2,4,6,8,10,12,14\} $\\
\hline
$ 30030 $&$ \{2,4,6,8,10,12,14,16,18\} $\\
\hline
\end{tabular}
\label{eventable}
\end{table}

For other choices of $k$, the exception sets $S_k$ can be readily computed using Theorem \ref{eventhm} and Corollary \ref{evencor}. 

\subsection{Results for general $n$}
We now detail our computations and results for general $n$. In this case we only consider odd $k$ since there are infinitely many odd $n$ such that $R_k(n)=0$ (and thus $\overline{R}_k(n)=0$) if $k$ is even (see Section \ref{sectasym}). The case when $n$ is odd is generally more difficult than that when $n$ is even. As a result, the following results are weaker than those in Section \ref{evenresults}.

\begin{lemma}\label{oddcomp}
    Let $n\geq 8\cdot 10^9$ and $k$ be odd and square-free with $1\leq k\leq 10^5$. Then for any $C\in(0,1/2)$ with $n^C>\sqrt{10^5}$,
    \begin{align}\label{oddeq}
        \frac{\overline{R_k}(n)}{n}>c&\prod_{q\mid k}\left(1-\frac{q-1}{q^2-q-1}\right)-\frac{1}{\log(n)}\sum_{\substack{d\mid k}}\sum_{e\mid d}\sum_{\substack{a\leq\sqrt{10^5}\sqrt{e/d}\\(a,d)=e}}c_{\theta}\left(da^2/e\right)\mu^2(a)\notag\\
        &-\left(\frac{1+2C}{1-2C}\right)\sum_{d\mid k}\sum_{e\mid d}\frac{1}{\varphi\left(d/e\right)}\sum_{\substack{a>\sqrt{10^5}\sqrt{e/d}\\(a,d)=e}}\frac{\mu^2(a)}{\varphi(a^2)}\notag\\
        &-\log(n)\left(\sum_{d\mid k}\sum_{e\mid d}\left(n^{-\frac{1}{2}}\left(\frac{1}{e}-\frac{1}{d}\right)+\frac{1}{\sqrt{de}}n^{-C}+n^{-2C}\right)\right)\notag\\
        &-\frac{\log(k)}{n}-\frac{\log(n)}{n}.
    \end{align}
    Here, we take $c_{\theta}(1)$ to be $9.5913\cdot 10^{-4}$ as per \cite[Table 15]{broadbent2021sharper}.
\end{lemma}
\begin{proof}
    Analogous to that of Lemma \ref{evencomp}. 
\end{proof}

Since we cannot use the Goldbach verification \cite{e2014empirical} for odd $n$, the computations became very cumbersome unless further restrictions were placed on $k$. For this reason, we only obtained a completely explicit result for odd $k$ with at most 2 prime factors.
\begin{theorem}[Equivalent to Theorem \ref{firstoddthm}] \label{oddthm}
    Let $k$ be odd with at most two prime factors and $1\leq k\leq 10^5$. Then for all $n \geq 36$, we have $\overline{R}_k(n)>0$.
\end{theorem}
To prove this result, first note that by choosing $C = 0.37$ and applying Lemma \ref{oddcomp}, the right-hand side of (\ref{oddeq}) is positive for $n \ge 8 \cdot 10^9$ and our choices of $k$. For even $40\leq n\leq 8\cdot 10^9$, the result follows from Corollary \ref{evencor}. The cases $n=36$ and $n=38$ are easily verified by hand. For instance, three representations of $36$ as the sum of a prime and a squarefree number are 
\begin{align*}
    36 &= 2 + 34, \\
    36 &= 3 + 33, \\
    36 &= 5 + 31.
\end{align*}
Since $31$, $33$ and $34$ are all mutually coprime, at least one these square-free numbers will be coprime to $k$, since $k$ has at most two prime factors. For odd $n \le 8\cdot10^9$, we verify using Python computations\footnote{See the file odd\_odd\_computations.py in \url{https://bit.ly/3m9QYlk}.}.

We employ an algorithm similar to \cite{francis2020additive}. We first compute the set $S$ of square-free numbers less than or equal to $1.6\cdot10^8$, excluding 1. We have $\overline{R}_k(n)>0$ (when $k$ has at most two prime factors) if we can find three distinct representations of $n$ as a sum of a prime and a square-free number, i.e.
\begin{align*}
    n &= p_1 + \eta_1, \\
    n &= p_2 + \eta_2, \\
    n &= p_3 + \eta_3,
\end{align*}
with $(\eta_1,\eta_2) \le 2$, $(\eta_1,\eta_3) \le 2$, and $(\eta_2,\eta_3) \le 2$. In order to do so, we divide the numbers $601 \le n \le 8\cdot10^9$ into 100 intervals $I_{1 \le \ell \le 100}$, roughly of the size $8\cdot10^7$. For each interval $I_{\ell}$, we compute the 100 largest primes $p_1 \le p_2 \le \ldots \le p_{100}$ in the interval $I_{\ell-1}$. For $I_1$, this would be the 100 largest primes $\le 600$. Subsequently, we check if $n-p_j$, starting from $j = 1$, is in $S$. Once we find at least three representations, we check if there is a combination of three square-free numbers that satisfies the above-mentioned property of their pairwise greatest common divisor being less than or equal to 2. We continue until such a triplet is found. For all $n \ge 601$, at least three such representations were found. This computation was done on Gadi, an HPC cluster at NCI Australia using 192 cores of Intel Xeon Cascade Lake processors and utilising 19.19 GB of memory (RAM). It took approximately 2 hours 53 minutes for the computations to terminate. For $n \le 600$, a brute-force algorithm was employed and the largest exception found was $n = 35$ for $k = 33$.

\begin{remark}
    The computations for $k$ prime were also performed by Francis and Lee \cite{francis2020additive}. However, we thought it would be instructive to reproduce their result.
\end{remark}

If $k$ has more than 2 prime factors we can still obtain an explicit result, albeit for much larger $n$.
\begin{theorem}\label{bigoddthm}
    Let $k$ be odd and square-free with $1\leq k\leq 10^5$. Then for all $n\geq 10^{25}$, we have $\overline{R}_k(n)>0$.
\end{theorem}
\begin{proof}
    In reference to Lemma \ref{oddcomp}, we choose $C = 0.2$. The value of $c_{\theta}(1)$ can also be reduced to $6.3417\cdot 10^{-9}$ since $n\geq 10^{25}$. With these parameters, the right-hand side of \eqref{oddeq} is positive for all $n\geq 10^{25}$ as required.  
\end{proof}
Presumably Theorem \ref{bigoddthm} holds for smaller values of $n$, however, significant computations would be required to find representations for each $n<10^{25}$.

\section{Further conditions on the square-free number}\label{sectext}
Our results so far have been concerned with the function $R_k(n)$ counting representations of $n$ as the sum of a prime and a square-free number $\eta$ with $(\eta,k)=1$. More generally though, we can consider the function
\begin{equation}\label{rkldef}
    R_k^\ell(n):=\sum_{\substack{p\leq n\\(n-p,k)=1\\ p\equiv n\:(\ell)}}\mu^2(n-p)\log(p).
\end{equation}
In particular, $R_k^\ell(n)$ is the weighted number of representations of $n$ as the sum of a prime $p$ and a square-free number $\eta$ with $(\eta,k)=1$ \textbf{and} $\ell\mid\eta$. One should take $(k,\ell)=1$ otherwise $R_k^\ell(n)$ would be zero for all $n$. Using our bounds from Section \ref{sectbounds}, we can also obtain bounds and asymptotics for $R_k^\ell(n)$.
\begin{proposition}\label{propklbounds}
    Let $\ell$ and $k$ be positive square-free numbers with $(\ell,k)=1$. Define
    \begin{align*}
        A_k^\ell(n)&:=A_k(n)\prod_{r\mid\ell}\frac{r-1}{r^2-r-1}\\
        &=cn\prod_{p\mid n}\left(1+\frac{1}{p^2-p-1}\right)\prod_{q\mid k}\left(1-\frac{q-1}{q^2-q-1}\right)\prod_{r\mid\ell}\frac{r-1}{r^2-r-1}
    \end{align*}
    where $p$, $q$ and $r$ are primes. Keeping the notation of Lemma \ref{delemma}, if $(\ell,n)=1$, $k_n=k/(k,n)$ and $n\geq n_0$ for each choice of $e$ and $\ell d$ with $e\mid d\mid k_n$, then
    \begin{equation}\label{rklbounds}
        \left|R_k^\ell(n)-A_{k_n}^\ell(n)\right|<\sum_{d\mid k_n}\sum_{e\mid \ell d}E_{\ell d,e}(n)-\log((k,n)).
    \end{equation}
\end{proposition}
\begin{proof}
    We assume $(k,n)=1$ and argue similarly to the proof of Proposition \ref{propcoprime}. The argument for $(k,n)>1$ is the same as in the proof of Proposition \ref{propnocoprime}.
    
    So, let $(k,n)=1$ thereby reducing \eqref{rklbounds} to
    \begin{equation*}
        \left|R_k^\ell(n)-A_k^\ell(n)\right|<\sum_{d\mid k}\sum_{e\mid \ell d}E_{\ell d,e}(n).
    \end{equation*}
    Define $E_k^\ell(n):=\sum_{d\mid k}\sum_{e\mid \ell d}E_{\ell d,e}(n)$. We will only prove $R_k^\ell(n)>A_k^\ell(n)-E_k^\ell(n)$ since the other inequality $R_k^\ell(n)<A_k^\ell(n)+E_k^\ell(n)$ follows from almost identical reasoning. Now,
    \begin{align*}
        R_k^\ell(n)&=R_1^\ell(n)-\sum_{\substack{p\leq n\\(n-p,k)>1\\p\equiv n\: (\ell)}}\mu^2(n-p)\log(p)\\
        &=\sum_{d\mid k}\mu(d)\sum_{\substack{p\leq n\\p\equiv n\:(\ell d)}}\mu^2(n-p)\log(p)\\
        &=\sum_{d\mid k}\mu(d)\sum_{e\mid \ell d}R_{\ell d,e}(n)\\
        &>\sum_{d\mid k}\mu(d)\sum_{e\mid \ell d}A_{\ell d,e}(n)-\sum_{d\mid k}\sum_{e\mid \ell d}E_{\ell d,e}(n)\\
        &=cn\prod_{p\mid n}\left(1+\frac{1}{p^2-p-1}\right)\prod_{q\mid k}\left(1-\frac{\varphi(q)}{q^2-q-1}\right)\prod_{r\mid\ell}\frac{\varphi(r)}{r^2-r-1}-E_k^\ell(n)\\
        &=A_k^\ell(n)-E_k^\ell(n)
    \end{align*}
    as required.
\end{proof}
\begin{theorem}\label{asympklthm}
    Keep the notation from Proposition \ref{propklbounds} and let $(n,\ell)=1$. If $k$ is odd,
    \begin{equation*}
        \frac{R_k^\ell(n)}{n}\sim c\prod_{p\mid n}\left(1+\frac{1}{p^2-p-1}\right)\prod_{q\mid k_n}\left(1-\frac{q-1}{q^2-q-1}\right)\prod_{r\mid\ell}\frac{r-1}{r^2-r-1}>0.
    \end{equation*}
    And if $k$ is even
    \begin{equation*}
        \frac{R_k^\ell(2n)}{2n}\sim c\prod_{p\mid 2n}\left(1+\frac{1}{p^2-p-1}\right)\prod_{q\mid k_{2n}}\left(1-\frac{q-1}{q^2-q-1}\right)\prod_{r\mid\ell}\frac{r-1}{r^2-r-1}>0.
    \end{equation*}
\end{theorem}
\begin{proof}
    Repeat the argument given in Theorem \ref{asympthm}, using Proposition \ref{propklbounds} as opposed to Proposition \ref{propnocoprime}.
\end{proof}
\begin{remark}
    In both Proposition \ref{propklbounds} and Theorem \ref{asympklthm} we required that $n$ was coprime to $\ell$. To illustrate why this is necessary, suppose that $(n,\ell)>1$. We are then looking at representations $n=p+\eta$ where $p$ is prime and $\eta$ is a square-free number divisible by $\ell$. If any representation exists, we require $p=(n,\ell)$ since $p=n-\eta$. As a result $\eta=n-(n,\ell)$. But for any $\ell$ there are infinitely many $n$ with $(n,\ell)>1$ and $n-(n,\ell)$ not square-free.
\end{remark}

One could use Proposition \ref{propklbounds} to obtain results similar to Theorem \ref{kthm} for other Goldbach-like problems. For instance, Li \cite[Theorem 1]{li2019representation} showed that every sufficiently large odd integer can be written as the sum of a prime and twice the product of at most 2 distinct primes. If Li's result were made explicit, one could repeat our computations in Section \ref{sectcomp} for $R_k^2(n)$, and use the argument from \cite[\S 10]{BJV22} to obtain a version of Theorem \ref{kthm} for odd numbers.

\section*{Acknowledgements}
We thank Matteo Bordignon, Forrest Francis, Ethan Lee and our supervisor Tim Trudgian for their helpful conversations. We also thank National Computational Infrastructure (NCI) Australia which is supported by the Australian Government and UNSW Canberra for computational resources.

\printbibliography
\end{document}